\documentclass[12pt]{l4dc2023}

\usepackage{enumerate}
\usepackage[shortlabels]{enumitem}

\DeclareMathOperator*{\argmin}{arg\,min}
\usepackage{tikz}
\usepackage{xcolor}
\usepackage{siunitx}
\usepackage{graphicx}

\newtheorem{assumption}[theorem]{Assumption}

\title[Forward gradient in first order online optimisation]{First order online optimisation using forward gradients in over-parameterised systems}
\usepackage{times}

\author{%
 \Name{Behnam Mafakheri} \Email{mafakherib@unimelb.edu.au}\\
 \addr Department of Electrical and Electronic Engineering, The University of Melbourne, Melbourne,Victoria, 3050, Australia
 \AND
 \Name{Iman Shames} \Email{iman.shames@anu.edu.au}\\
 \addr CIICADA Lab, Australian National University, Canberra, 2601, Australia%
 \AND
 \Name{Jonathan H. Manton} \Email{j.manton@ieee.org}\\
 \addr Department of Electrical and Electronic Engineering, The University of Melbourne, Melbourne,Victoria, 3050, Australia%
}

\begin{document}
\maketitle

\begin{abstract}
    The success of deep learning over the past decade mainly relies on gradient-based optimisation and backpropagation. This paper focuses on analysing the performance of first-order gradient-based optimisation algorithms, gradient descent and proximal gradient, with time-varying non-convex cost function under (proximal) Polyak-{\L}ojasiewicz condition. Specifically, we focus on using the forward mode of automatic differentiation to compute gradients in the fast-changing problems where calculating gradients using the backpropagation algorithm is either impossible or inefficient. Upper bounds for tracking and asymptotic errors are derived for various cases, showing the linear convergence to a solution or a neighbourhood of an optimal solution, where the convergence rate decreases with the increase in the dimension of the problem. We show that for a solver with constraints on computing resources, the number of forward gradient iterations at each step can be a design parameter that trades off between the tracking performance and computing constraints.
\end{abstract}


\section{Introduction}
Gradient-based optimisation is the core of most machine learning algorithms \citep{goodfellow2016deep}. Backpropagation, the reverse mode of Automatic Differentiation (AD) algorithms, has been the main algorithm for computing gradients. \citet{baydin2022gradients} has argued that in training neural networks, one can calculate gradients based on directional derivatives, which is faster than backpropagation. An unbiased estimate of gradients, named forward gradients, can be calculated in a single forward run of the function and speeds up training up to twice as fast in some examples. Forward differentiation is a mode of automatic differentiation algorithm. It is based on the mathematical definition of dual numbers (see \citet{baydin2018automatic} for a review).

Given a function $f: \mathbb{R}^n \to \mathbb{R}$, the \emph{forward gradient} at point $\boldsymbol{x}$ is defined as
\begin{align}
    \boldsymbol{g}(\boldsymbol{x}, \boldsymbol{u}) = \langle \nabla f(\boldsymbol{x}), \boldsymbol{u}\rangle \boldsymbol{u},
\end{align}
where $\boldsymbol{u}$ is a random vector with zero mean and unit variance. Note that $\langle \nabla f(\boldsymbol{x}), \boldsymbol{u} \rangle$ is the directional derivative of $f$ at point $\boldsymbol{x}$ in the direction $\boldsymbol{u}$. The forward gradient, $g(\boldsymbol{x}, \boldsymbol{u})$, is an unbiased estimator for $\nabla f(\boldsymbol{x})$ as long as the components of direction $\boldsymbol{u}$ are independent and sampled from a distribution with zero mean and unit variance; $\left[\mathbb{E}(g(\boldsymbol{x}, \boldsymbol{u}))\right]_i = \mathbb{E}(u_i \sum_{j=1}^n (\nabla f(\boldsymbol{x}))_j u_j) = (\nabla f(\boldsymbol{x}))_i$, where $a_i$ is the $i$-th element of vector $\boldsymbol{a}$.

Deep neural networks can be considered as a particular case of the over-parameterised system of nonlinear equations in which there are potentially a staggering number of trainable parameters \citep{fedus2021switch}. In \citep{liu2022loss} authors argue that focusing on the convexity of the problem does not lead to a suitable framework for analysing such loss functions and instead show that the Ployak-{\L}ojasiewicz (PL) condition \citep{polyak1963gradient} is satisfied on the most of parameter space. This in turn explains fast convergence of Gradient Descent (GD) algorithm to a global minimum. 

Another area that has observed a flurry of activity is anchored at studying time-varying objective functions. These problems arise in many applications including but not limited to robotics \citep{zavlanos2012network}, smart grids \citep{dall2016optimal, tang2017real}, communication systems \citep{chen2011convergence} and signal processing \citep{yang2016online, vaswani2016recursive}. An example is when a machine learning problem's training data (problem parameters) are observed during the learning process. These can be considered as a sequence of optimisation problems and a solver may solve each problem completely before approaching the next problem, which is not feasible in cases where each instance of the problem is of large scale (considering the rate of change and the speed of solver) or involves communication over a network.  Therefore, a more efficient way may be to solve each problem partially (e.g. performing only a few gradient descent steps) before receiving the next problem. These algorithms are called \emph{running} method in \citet{simonetto2017time}. Consider a neural network that has been trained over a data set $\mathcal{D} = \{\boldsymbol{a}_i, b_i \}_{i=0}^{n-1}$. Let $\mathcal{L}_0(\boldsymbol{x}) = \frac{1}{2}\sum_{i=0}^{n-1} \|f(\boldsymbol{x} ; \boldsymbol{a}_i) - b_i \|^2$. Assume that, while deploying the pre-trained neural network, new samples of data arrive at a high rate and one needs to update the training variable of the network, $\boldsymbol{x}$. This is an over-parameterised online optimisation problem. We assume that only the latest $n$ data samples to update the variables are used. The corresponding cost functions is
\begin{align}
    \mathcal{L}_k(\boldsymbol{x}) = \frac{1}{2}\sum_{i=n_k}^{n_k+n-1} \| f(\boldsymbol{x} ; \boldsymbol{a}_i) - b_i \|^2,
\end{align}
where $n_k - n_{k-1}$ is the number of newly observed samples.
Note that we assume that the structure of the neural network does not change during time and we only try to update its variables after observing new data examples (See \citep{parisi2019continual} and references therein for a review on expandable neural networks). We assume
\begin{itemize}
    \item The algorithm uses the latest samples (because they represent the data distribution better);
    \item Only $n$ samples (e.g. because of memory limitations) can be used;
    \item Due to its fast computation, forward gradient is to be used in the optimisation steps.
\end{itemize}
Ideally, the solution generated by the method at each time step, $\boldsymbol{x}_k$, follows $\boldsymbol{x}_k^*$. In other words, the tracking error $\|\boldsymbol{x}_k - \boldsymbol{x}_k^* \|$ or the instantaneous sub-optimality gap $\|\mathcal{L}_k(\boldsymbol{x}_k) - \mathcal{L}_k(\boldsymbol{x}_k^*) \|$ either shrink or remains bounded as $k$ grows. We assume that $\ell$ optimisation step is applied to $\mathcal{L}_k$ before receiving the new batch of data, where $\ell$ is chosen based on the above constraints.

In this paper, we focus on the analysis of the performance of forward gradient-based method in over-parameterised time-varying settings characterised by the following cost function at time $k$: 
\begin{align}\label{composite_sum_timevar}
    \min_{\boldsymbol{x}\in \mathbb{R}^m} \ \{ \mathcal{L}_k(\boldsymbol{x}) := g_k(\boldsymbol{x}) + h_k(\boldsymbol{x}) \}.
\end{align}
Function $g_k$ is smooth and potentially nonconvex and $h_k(\boldsymbol{x})$ is possibly nonsmooth convex function that imposes some structure on the solution such as sparsity. We consider two cases with or without the presence of function $h$ in the cost function. Such problems arise in the online machine learning \citep{shalev2012online} and have been well studied under (strong) convexity assumption. The contributions of this paper are listed below:
\begin{itemize}
    \item We prove the linear convergence in expectation of forward-gradient method to the minimiser for smooth nonconvex PL objective functions (Theorem \ref{Exp_var_nest}). 
    \item We prove that the expected tracking error of the forward-gradient method in a time-varying nonconvex setting, under smoothness and PL assumptions, converges to a neighbourhood of zero at a linear rate (Theorem \ref{fwd_grad_Timevar}). We show that the number of iterations at each step, $\ell$, is a design parameter that trades off between the tracking performance and computing resources constraints.
    \item In the case where the objective cost function is the sum of a PL and a convex (possibly non-smooth) function, we prove that the expected tracking error of the proximal forward-gradient-based algorithms converges to a neighbourhood of zero for time-varying cost functions under proximal-PL condition (Theorem \ref{prox_fwd_grad_timevar}). Error depend on the problem dimension, but the rate does not. The proof for quadratic growth of the proximal-PL functions is also novel.
\end{itemize}
The proofs are described in detail in the technical report associated with this paper \citep{mafakheri2022order}.

The rest of the paper is organised as follows: in Section \ref{related_works}, we discuss related works, and in Section \ref{notation_prim}, we review notations, definitions, basic assumptions and some background knowledge. In Section \ref{fwd_gd_fixed_PL}, we prove the linear convergence of forward gradient algorithm for the PL objective functions. In Section \ref{time_var_PL}, we analyse the performance of gradient descent and proximal gradient algorithms under PL and Proximal-PL conditions. Section \ref{conclusion} concludes the paper.

\section{Related works}\label{related_works}
\emph{PL condition in over-parameterised nonlinear systems}: The nonlinear least square problem has been extensively studied in under-parameterised systems \citep{nocedal1999numerical}. In \citep{liu2022loss, belkin2021fit} the sufficient condition for satisfying the PL condition in a wide neural network with least square loss function and linear convergence of (S)GD have been studied and the relation with neural tangent kernel has been discussed \citep{jacot2018neural}.

\emph{Static cost functions:} In \citep{karimi2016linear} convergence of Randomised Coordinate Descent (RCD), Stochastic Gradient Descent, and SVRG algorithms under PL condition and Proximal coordinate descent algorithm under Proximal-PL has been studied. Inexact gradient descent algorithms have been analysed under a wide range assumptions (See \citep{khaled2020better} and references therein). In \citep{vaswani2019fast} $\mathcal{O}(1/k)$ convergence of SGD under weaker condition (Weak Growth Condition) has been proved.

\emph{Time-varying cost function:} A closely related line of research to this paper is online convex optimisation (OCO) which was originally introduced in seminal work by \citet{zinkevich2003online}. Since the focus of this work is on nonconvex optimisation first order algorithm with non-exact gradients, we omit to discuss the vast literature on the online exact first order algorithms under strong convexity and smoothness assumptions (see \citet{madden2021bounds, selvaratnam2018numerical} and references therein). \citet{yang2016tracking, devolder2014first, schmidt2011convergence, villa2013accelerated} analysed the regret bounds for online first order optimisation algorithms under (strong) convexity and smoothness with noisy gradients. Regret analysis of online stochastic fixed and time-varying optimisation have been done in \citep{hazan2016introduction} and \citep{cao2020online, shames2020online}. \citet{ospina2022feedback} and \citet{kim2021convergence} have analysed the online (proximal) gradient methods with sub-Weibull gradient error under strong convexity and PL conditions, respectively. In \citet{michael2022gradient} authors have proposed a method for estimating the gradient in a time-varying setting with zeroth-order distributed optimisation algorithm.

\section{Notation and Preliminaries}\label{notation_prim}
Throughout this report, we show the set of real numbers as $\mathbb{R}$. For vectors $\boldsymbol{a}, \boldsymbol{b}\in \mathbb{R}^n$ the Euclidean inner product and its corresponding norms are denoted by $\langle \boldsymbol{a} , \boldsymbol{b} \rangle$ and $\| \boldsymbol{a}\|$ respectively. We denote the open ball centered at $\boldsymbol{x}$ with radius $r$ by $B(\boldsymbol{x}, r)$. With $\pi_{\mathcal{X}}(\boldsymbol{x})$ we indicate the projection of vector $\boldsymbol{x}$ into set $\mathcal{X}$. The Euclidean distance to  a set $\mathcal{X} \in \mathbb{R}^n$ is defined by $dist(\boldsymbol{x}, \mathcal{X}):= \inf_{\boldsymbol{u} \in \mathcal{X}} \| \boldsymbol{u} - \boldsymbol{x} \|$ for every $\boldsymbol{x}\in \mathbb{R}^n$. We denote the mapping of projection over set $\mathcal{X}$ with $\pi_{\mathcal{X}}$. With $\mathcal{X}^*$ we denote the set of minimisers of a function. For a matrix $A\in \mathbb{R}^{m\times n}$, $\| A\|$ and $\| A\|_F$ are the spectral and Frobenius norms of the matrix, respectively. We use $\mathcal{D} F(\cdot)$ and $\mathcal{D}^2 F(\cdot)$ to show the derivative and Hessian of a function $F:\mathbb{R}^n \to \mathbb{R}$, respectively. We use $\widetilde{\mathcal{O}}(\cdot)$ notation by ignoring the logarithmic factors in big-O notation, i.e. $\widetilde{\mathcal{O}}(f(n)) = \mathcal{O}(f(n)\log f(n))$.

In what follows, we introduce some definitions, assumptions and lemmas that will be utilised in the paper.
\begin{definition}
The function $\mathcal{L}: \mathbb{R}^n \to \mathbb{R}$ is $\beta-smooth$ if it is differentiable and
\begin{align*}
    \|\nabla \mathcal{L}(\boldsymbol{y}) - \nabla \mathcal{L}(\boldsymbol{x}) \| \leq \beta\| \boldsymbol{y} - \boldsymbol{x} \|, \quad \forall \boldsymbol{x}, \boldsymbol{y} \in \mathbb{R}^n.
\end{align*}
\end{definition}
\begin{definition}
The function $\mathcal{L}: \mathbb{R}^n \to \mathbb{R}$ satisfies the Polyak-{\L}ojasiewicz (PL) condition on a set $\mathcal{X}$ with constant $\mu$ if
\begin{align}
       \frac{1}{2} \|\nabla \mathcal{L}(\boldsymbol{x}) \|^2 \geq \mu (\mathcal{L}(\boldsymbol{x}) - \mathcal{L}^*), \quad \forall \boldsymbol{x} \in \mathcal{X}.
\end{align}
\end{definition}
\begin{lemma} \label{descent_lemma} 
(Descent Lemma \citet[Proposition~A.24]{bertsekas1997nonlinear}) Let the function $\mathcal{L}: \mathbb{R}^n \to \mathbb{R}$ be a $\beta_{\mathcal{L}}$-smooth function. Then for every $\boldsymbol{x}$, $\boldsymbol{y} \in \mathbb{R}^n$ and for every $\boldsymbol{z} \in [\boldsymbol{x}, \boldsymbol{y}] := \{(1-\alpha)\boldsymbol{x} + \alpha \boldsymbol{y} : \alpha \in [0,1] \}$ the following holds
\begin{align*}
    \mathcal{L}(\boldsymbol{y}) \leq \mathcal{L}(\boldsymbol{x}) + \langle \nabla \mathcal{L}(\boldsymbol{z}), \boldsymbol{y} - \boldsymbol{x}\rangle + \frac{\beta_{\mathcal{L}}}{2} \|\boldsymbol{y} - \boldsymbol{x} \|^2.
\end{align*}
\end{lemma}

\begin{lemma} \label{PL-smooth_ineq}
(\citet{karimi2016linear}) Let $\mathcal{L}: \mathbb{R}^n \to \mathbb{R}$ be a $\beta$-smooth and $\mu-PL$ function. Then
\begin{align*}
    \frac{\mu}{2}\|\boldsymbol{x} - \pi_{\mathcal{X}^*}(\boldsymbol{x}) \|^2\leq \mathcal{L}(\boldsymbol{x}) - \mathcal{L}^* \leq \frac{\beta}{2}\|\boldsymbol{x} - \pi_{\mathcal{X}^*}(\boldsymbol{x})  \|^2, \quad \forall \boldsymbol{x}.
\end{align*}
\end{lemma}
\begin{definition} \label{proximal PL} (Proximal-PL condition \citet{karimi2016linear}) Let $\mathcal{L}(\boldsymbol{x}) = g(\boldsymbol{x})+ h(\boldsymbol{x})$ where $g$ is $\beta-$smooth and $h$ is a convex function. The function $f$ satisfies $\mu-$Proximal PL-condition if the following holds
\begin{align}\label{prox PL ineq}
    \frac{1}{2}\mathcal{D}_h(\boldsymbol{x}, \beta) \geq \mu(\mathcal{L}(\boldsymbol{x}) - \mathcal{L}^*), \quad \forall \boldsymbol{x},
\end{align}
where
\begin{align}
    \mathcal{D}_h(\boldsymbol{x}, \alpha):= -2\alpha \min_{\boldsymbol{y}} \left\{ \langle \nabla g(\boldsymbol{x}), \boldsymbol{y}-\boldsymbol{x} \rangle + \frac{\alpha}{2}\|\boldsymbol{y}-\boldsymbol{x} \|^2 + h(\boldsymbol{y}) - h(\boldsymbol{x})\right\}
\end{align}
\end{definition}

\subsection {Over-parameterised systems and PL condition}
For a system of $n$ nonlinear equations
        \begin{align*}
            f(\boldsymbol{x}; \boldsymbol{a}_i) = b_i, \quad i=1,2, \dots, n,
        \end{align*}
    where $\{\boldsymbol{a}_i, b_i \}_{i=1}^n$ is the set of model parameters and one aims at finding $\boldsymbol{x}\in \mathbb{R}^m$ that solves the system of equations. Aggregating all equations in a single map amounts to
        \begin{align}\label{Single-map}
            \mathcal{F}(\boldsymbol{x}) = \boldsymbol{y}, \ \text{where} \ \boldsymbol{x}\in \mathbb{R}^m, \mathcal{F}(\cdot):\mathbb{R}^m \to \mathbb{R}^n.
        \end{align}
The system in \eqref{Single-map} is solved through minimising a certain loss function $\mathcal{L}(\boldsymbol{x})$:
        \begin{align}
            \mathcal{L}(\boldsymbol{x}):=\frac{1}{2} \|\mathcal{F}(\boldsymbol{x}) - \boldsymbol{b} \|^2=\frac{1}{2} \sum_{i=1}^n (f(\boldsymbol{x}; \boldsymbol{a}_i) - b_i)^2.
        \end{align}
This problem has been studied extensively in under-parameterised setting (where $m<n$). We refer to exact solution of \eqref{Single-map} as interpolation.
\begin{definition}
Let $D\mathcal{F}(\boldsymbol{x})$ be the differential of the map $\mathcal{F}$ at $\boldsymbol{x}$ which can be represented as a $n\times m$ matrix. The tangent Kernel of $\mathcal{F}$ is defined as a $n\times n$ positive semidefinite matrix $K(\boldsymbol{x}):=D\mathcal{F}(\boldsymbol{x})D\mathcal{F}^T(\boldsymbol{x})$.
\end{definition}
\begin{definition}
We say that a non-negative function $\mathcal{L}$ satisfies $\mu-PL^*$ condition on a set $\mathcal{X}\in \mathbb{R}^m$ for $\mu >0$, if
\begin{align}\label{PL}
   \frac{1}{2} \|\nabla \mathcal{L}(\boldsymbol{x}) \|^2 \geq \mu \mathcal{L}(\boldsymbol{x}), \quad \forall \boldsymbol{x} \in \mathcal{X}.
\end{align}
\end{definition}
\begin{remark}
\begin{enumerate}[(a)]
    \item if a non-negative function satisfies $\mu-PL^*$, then it will satisfy $\mu-PL$ condition too, that is
        \begin{align}
            \frac{1}{2}\|\nabla \mathcal{L} (\boldsymbol{x} \|^2 \geq \mu(\mathcal{L}(\boldsymbol{x}) - \mathcal{L}^*)
        \end{align}
    \item Every stationary point of a PL function, is a global minimum.
\end{enumerate}

\end{remark}
The following results are of great importance in the later discussions.
\begin{theorem}\label{PL_convergence} (\citet{polyak1963gradient})
If $\mathcal{L}(\boldsymbol{x})$ is lower bounded with $\beta$-Lipschitz continuous gradients and satisfies $\mu-PL$ condition in the region $\mathcal{X} = \overline{B(\boldsymbol{x_0}, \rho)}$ where $\rho > \frac{\sqrt{2\beta(\mathcal{L}(\boldsymbol{x}_0) - \mathcal{L}^*)}}{\mu}$, then there exists a global minimum point $x^* \in \mathcal{X}$ and the gradient descent algorithm with a small enough step-size ($\frac{1}{\beta}$) converges to $x^*$ in a linear rate.
\end{theorem}
\begin{theorem}(\citet{liu2022loss})
If $\mathcal{F}(\boldsymbol{x})$ is such that $\lambda_{min}(K(\boldsymbol{x}))\geq \mu > 0$ for all $\boldsymbol{x} \in \mathcal{X}$, then the square loss function $\mathcal{L}(\boldsymbol{x}) = \frac{1}{2}\|\mathcal{F}(\boldsymbol{x})- \boldsymbol{y} \|^2$ satisfies $\mu-PL^*$ condition on $\mathcal{X}$.
\end{theorem}
Note that $K(\boldsymbol{x})=D\mathcal{F}(\boldsymbol{x})D\mathcal{F}^T(\boldsymbol{x})$ implies that $rank(K(\boldsymbol{x})) = rank(D\mathcal{F}(\boldsymbol{x}))$. In an over-parameterised system, $m>n$, starting from a random initial point there is a high probability that $D\mathcal{F}(\boldsymbol{x})$ is full rank. However, from Theorem \ref{PL_convergence}, one needs $\mu-PL^*$ condition in a large enough ball with radius $\mathcal{O}(\frac{1}{\mu})$ for linear convergence of the GD algorithm. For establishing such conditions, one intuitively can expect that in the cases of $\mathcal{C}^2$ function that if Hessian (curvature) is small enough in the neighbourhood of a point, then the Tangent Kernel should be almost constant and therefore, the conditions for linear convergence of the GD algorithm will be satisfied. It turns out that in the case of highly over-parameterised Neural Networks (wide NNs) with linear output layer, the Hessian matrix will have arbitrary small spectral norm (a transition to linearity). This is formulated as follows:
\begin{theorem}(\citet{liu2020linearity})
For an $L$ layer neural network with a linear output layer and minimum width $m$ of the hidden layers, for any $R>0$ and any $\boldsymbol{x}\in B_{R}(\boldsymbol{x}_0)$. the Hessian spectral norm satisfies the following with a high probability
\begin{align}
    \|\mathcal{D}^2 \mathcal{F}(\boldsymbol{x}) \| = \widetilde{\mathcal{O}}\left (\frac{R^{3L}}{\sqrt{m}} \right).
\end{align}
\end{theorem}
\begin{remark}\citet[Thm.~3]{liu2020linearity}
With a nonlinear output layer, $\boldsymbol{x} \mapsto \varphi(\boldsymbol{x})$, the square loss function will satisfy $\mu\kappa^2$-PL condition where $\kappa:=\inf_{\boldsymbol{x} \in B(\boldsymbol{x}_0, \rho)} \|\mathcal{D}\varphi(\mathcal{F}(\boldsymbol{x})) \|$.
\end{remark}
All in all, wide neural networks, as a particular case of over-parameterised systems, satisfy the $\mu-PL^*$ condition, which explains the fast convergence of (S)GD to a global minimum in square-loss problems. In the following sections, we theoretically analyse the performance of using forward gradient in the same setting (over-parameterised systems).

\section{Optimisation using forward gradient}
In this section, we analyse the performance of various gradient-based algorithms using forward gradient. 
\subsection {Forward gradient under PL condition for fixed cost functions} \label{fwd_gd_fixed_PL}

With the focus on the basic unconstrained optimisation problem
\begin{align}
    \min_{\boldsymbol{x}\in \mathbb{R}^m} \mathcal{L}(x).
\end{align}
Consider the following iterations where the solver is not fast enough to do one gradient calculation, but it is able to do $\ell$ forward gradient updates at time step $k$ 
\begin{align}\label{GD with directional derivative}
    \boldsymbol{x}_{k+1}^{(i+1)} = \boldsymbol{x}_{k+1}^{(i)} - \alpha_k^{(i)} \boldsymbol{v}_{k}(\boldsymbol{x}_{k+1}^{(i)}, U_{k+1}^{(i)}), \quad \textit{for}\  i=0,1, \dots, \ell -1,
\end{align}
where  $\boldsymbol{x}_{k+1}^{(\ell)} := \boldsymbol{x}_{k+1}$, $\boldsymbol{x}_{k+1}^{(0)} := \boldsymbol{x}_k$, $\alpha_k$ is step-size, $\boldsymbol{v}(\boldsymbol{x}, \boldsymbol{u}) = \langle \nabla \mathcal{L}(\boldsymbol{x}), \boldsymbol{u}\rangle \boldsymbol{u}$, and $U_k \sim \mathcal{N}(\boldsymbol{0}, I_m)$ are i.i.d random directions for all $k$ and $i$. It is extensively studied in \citet{nesterov2017random} where the following properties have been proved:
\begin{subequations}\label{Exp_var_nest}
\begin{align}
    \mathbb{E} \boldsymbol{v}(\boldsymbol{x}, U_k) &= \nabla \mathcal{L}(\boldsymbol{x})\\
    \mathbb{E}(\|\boldsymbol{v}(\boldsymbol{x}, U_k) \|^2) &\leq (m+4)\| \nabla \mathcal{L}(\boldsymbol{x})\|^2 \label{Var-Dir_Deriv}
\end{align}
\end{subequations}
In the following we analyse the convergence of this algorithm under PL condition. We refer the reader to the extended version of this paper \citep{madden2021bounds} for proofs.
\begin{theorem} \label{lemma_fix_fwd_gd}
Assume that function $\mathcal{L}$ is $\beta-$smooth, has a non-empty solution set $\mathcal{X}^*$, and satisfies $\mu-$PL condition. Consider the algorithm \eqref{GD with directional derivative} with a step-size of $\frac{1}{\beta(m+4)}$. If random vector $U_k$ is chosen from a standard normal distribution, i.e. and $U_k \sim \mathcal{N}(\boldsymbol{0}, I_m)$, then the algorithm has an expected linear convergence rate
\begin{align}
    \mathbb{E}\{\mathcal{L}(x_k) - \mathcal{L}^* \} \leq \left(1-\frac{\mu}{(m+4)\beta}\right)^k (\mathcal{L}(x_0) - \mathcal{L}^*)
\end{align}
\end{theorem}

\subsection {Time varying Optimisation} \label{time_var_PL}

In what follows, we prove a convergence property of online line-search methods for the objective functions that satisfy the following assumptions.
\begin{assumption}\label{PLcond} (Polyak-{\L}ojasiewicz Condition) There exist a scalar $\mu>0$ such that $\mathcal{L}_{k}(\boldsymbol{x})$ satisfies \eqref{PL} for all $k$  and for all $\boldsymbol{x} \in \mathbb{R}^m$.
\end{assumption}
We also need to quantify the speed with which the objective function varies in each step.
\begin{assumption} \label{drift} (Drift in Time): There exist non-negative scalars $\eta_0$ and $\eta^*$ such that $\mathcal{L}_{k+1}(\boldsymbol{x}) - \mathcal{L}_k(\boldsymbol{x}) \leq \eta_0$ for all $\boldsymbol{x} \in \mathbb{R}^m$ and $\mathcal{L}_{k+1}^* - \mathcal{L}_k^* \leq \eta^*$ for all $k$.

\end{assumption}
\begin{remark}
In an over-parameterised setting with non-linear least square loss function, one can argue that $\mathcal{L}_k^* = 0$, for all $k$. Therefore, $\eta^* = 0$ in this setting.
\end{remark} 
\subsubsection{Forward gradient convergence in Time varying setting}
In this part, we analyse the performance of the forward-gradient algorithm with a fixed step size for the time-varying nonconvex cost functions under PL and smoothness assumptions. We prove that the tracking error linearly converges to a neighbourhood of zero where the rate of convergence depends on the dimension of the problem and the asymptotic error is independent of the fact that we are not using exact gradients.
\begin{theorem}\label{fwd_grad_Timevar}
Assume that $\mathcal{L}_k$ is $\beta-$smooth for each $k$, and Assumptions \ref{PLcond} and \ref{drift} hold. If we use directional derivative algorithm with constant step-size $\alpha_k^{(i)} = \alpha < \frac{2}{\beta(m+4)}$ and random direction $U_k^{(i)}$ is chosen from a standard normal distribution, then 
\begin{align}\label{err_bdd_timevar}
    \mathbb{E}\{\| \boldsymbol{x}_{k+1} - \pi_{\mathcal{X}_{k+1}^*}(\boldsymbol{x}_{k+1}) \|^2\} &\leq \frac{\eta_0 + \eta^*}{\mu^2 \gamma \ell} + \frac{2}{\mu}(1-2\mu\gamma\ell)^k (\mathcal{L}_0(\boldsymbol{x}_0) - \mathcal{L}_0^*),
\end{align}
where $\eta_0$ and $\eta^*$ are as in Assumption \ref{drift}, $\gamma \in (0, \tilde{\gamma})$ where $\tilde{\gamma} = min \{\alpha(1-\frac{\beta}{2}(m+4)\alpha), \frac{1}{2\mu\ell} \}$, and $\mathcal{X}_k^*$ is the set of minimisers of $\mathcal{L}_k(\boldsymbol{x})$.
\end{theorem}
\begin{proof}

Using descent lemma and \eqref{GD with directional derivative} we have
\begin{align*}
    \mathcal{L}_{k}(\boldsymbol{x}_{k+1}^{(i+1)}) &\leq \mathcal{L}_k(\boldsymbol{x}_{k+1}^{(i)}) + \langle \nabla \mathcal{L}_{k}(\boldsymbol{x}_{k+1}^{(i)}), \boldsymbol{x}_{k+1}^{(i+1)} - \boldsymbol{x}_{k+1}^{(i)} \rangle + \frac{\beta}{2}\|\boldsymbol{x}_{k+1}^{(i+1)} - \boldsymbol{x}_{k+1}^{(i)} \|^2 \\
    & =  \mathcal{L}_k(\boldsymbol{x}_{k+1}^{(i)}) -\alpha_k^{(i)} \langle \nabla \mathcal{L}_{k}(\boldsymbol{x}_{k+1}^{(i)}), \boldsymbol{v}_k(\boldsymbol{x}_{k+1}^{(i)}, U_{k+1}^{(i)}) \rangle + \frac{\beta}{2}(\alpha_k^{(i)})^2\|\boldsymbol{v}_k(\boldsymbol{x}_{k+1}^{(i)}, U_{k+1}^{(i)}) \|^2.
\end{align*}
By taking conditional expectation given $\boldsymbol{x}_{k+1}^{(i)}$ with respect to $U_{k+1}^{(i)}$, and using \eqref{Exp_var_nest} we obtain
\begin{align}
    \mathbb{E}\{\mathcal{L}_k(\boldsymbol{x}_{k+1}^{(i+1)}) \mid \boldsymbol{x}_{k+1}^{(i)} \} &\leq \mathcal{L}_k(\boldsymbol{x}_{k+1}^{(i)}) -\alpha_{k}^{(i)}  \|\nabla \mathcal{L}_{k}(\boldsymbol{x}_{k+1}^{(i)}) \|^2 +\frac{1}{2}\beta (\alpha_k^{(i)})^2(m+4)\|\nabla \mathcal{L}_{k}(\boldsymbol{x}_{k+1}^{(i)}) \|^2 \nonumber\\
   & = \mathcal{L}_k(\boldsymbol{x}_{k+1}^{(i)}) -\underbrace{\alpha_{k}^{(i)}(1-\frac{\beta}{2}(m+4)\alpha_{k}^{(i)})}_{:=\gamma_{k}^{(i)}}\|\nabla\mathcal{L}_{k}(\boldsymbol{x}_{k+1}^{(i)}) \|^2 \label{descent_lemma_implication}.
\end{align}
We now fix $\epsilon_1  <\alpha_{k}^{(i)}= \alpha < \frac{2}{\beta(m+4)} - \epsilon_2$ and $\gamma_{k}^{(i)} = \gamma >0$ for some $\epsilon_1 >0$ and $\epsilon_2 >0$. Assuming that $\mathcal{L}_k$ satisfies $\mu-PL$ condition at $\boldsymbol{x}_{k+1}^{(i)}$, for all $i$ and $k$, we have that
\begin{align}\label{bound_one_iteration}
        \mathbb{E}\{\mathcal{L}_k(\boldsymbol{x}_{k+1}^{(i+1)}) \mid \boldsymbol{x}_{k+1}^{(i)} \} &\leq \mathcal{L}_k(\boldsymbol{x}_{k+1}^{(i)}) -2\mu \gamma (\mathcal{L}_k(\boldsymbol{x}_{k+1}^{(i)}) - \mathcal{L}_k^*), \quad \text{for } i=0, \dots, \ell-1.
\end{align}
Applying the above inequality recursively for $i=0,1, \dots, \ell-1$ amounts to
\begin{align*}
    \mathbb{E}\{\mathcal{L}_k(\boldsymbol{x}_{k+1}) \mid \boldsymbol{x}_{k} \} &\leq \mathcal{L}_k(\boldsymbol{x}_{k}) -2\mu \gamma \sum_{i=0}^\ell \left(\mathcal{L}_k(\boldsymbol{x}_{k+1}^{(i)}) - \mathcal{L}_k^*\right)\\
    &\stackrel{(a)}{\leq} \mathcal{L}_k(\boldsymbol{x}_{k}) -2\mu \gamma \ell \left(\mathcal{L}_k(\boldsymbol{x}_{k}) - \mathcal{L}_k^*\right) ,
\end{align*}
where in (a) we have used the \eqref{descent_lemma_implication}. By adding and subtracting terms we have
\begin{align*}
    \mathbb{E}\left\{\mathcal{L}_{k+1}(\boldsymbol{x}_{k+1}) - \mathcal{L}_{k+1}^* | \boldsymbol{x}_k\right\} = & \mathbb{E}\left\{\mathcal{L}_{k+1}(\boldsymbol{x}_{k+1}) - \mathcal{L}_{k}(\boldsymbol{x}_{k+1}) | \boldsymbol{x}_k\right\} + \mathbb{E}\left\{\mathcal{L}_{k}(\boldsymbol{x}_{k+1}) - \mathcal{L}_{k}(\boldsymbol{x}_{k}) | \boldsymbol{x}_k\right\} \\
    &+ \left(\mathcal{L}_k(\boldsymbol{x}_k) - \mathcal{L}_k^*\right) + (\mathcal{L}_k^* - \mathcal{L}_{k+1}^*) \\
    &\stackrel{(b)}{\leq} \eta_0 + \eta^* + (1-2\gamma \mu \ell) (\mathcal{L}_k(\boldsymbol{x}_k) - \mathcal{L}_k^*),
\end{align*}
where in (b) we have used the bounds on drift and \eqref{bound_one_iteration}.
By using the tower rule of expectations and induction we have
\begin{align}
    \mathbb{E}\{\mathcal{L}_k(\boldsymbol{x}_{k+1}) - \mathcal{L}_{k+1}^*\} &\leq (\eta_0 + \eta^*) \sum_{i=0}^{k} (1-2\gamma \mu \ell)^{k-i} + (1-2\mu\gamma\ell)^k (\mathcal{L}_0(\boldsymbol{x}_0) - \mathcal{L}_0^*) \\
    &\leq \frac{\eta_0 + \eta^*}{2\mu \gamma \ell} + (1-2\mu\gamma\ell)^k (\mathcal{L}_0(\boldsymbol{x}_0) - \mathcal{L}_0^*).
\end{align}
Invoking \eqref{PL-smooth_ineq} and choosing $\gamma \in (0, \tilde{\gamma})$ completes the proof.

\end{proof}

\begin{remark}
In the case of a pre-trained neural network where the term $\mathcal{L}_0(\boldsymbol{x}_{0}) - \mathcal{L}_{0}^*$ is small (or zero), the slower convergence rate in higher dimensions is not a problem while the asymptotic error bound, the first term in \ref{err_bdd_timevar}, is of great interest. By choosing the parameter $\ell$ properly, one can get better bounds based on the speed  of its solver or equivalently the number of iterations at each step. 
\end{remark}
\begin{remark} Under the hypotheses of Theorem \ref{fwd_grad_Timevar} the following holds:
\begin{align}
    \limsup_{k\to \infty}{\mathbb{E}\{\mathcal{L}_k(\boldsymbol{x}_{k+1}) - \mathcal{L}_{k+1}^*\} }&\leq \frac{\eta_0 + \eta^*}{2\mu\gamma\ell}.
\end{align}
This in turn results in $\limsup_{k \to \infty} \| \boldsymbol{x}_k - \pi_{\mathcal{X}_{k}^*}(\boldsymbol{x}_{k}) \|^2 \leq \frac{\eta_0 + \eta^*}{\mu^2\gamma\ell}$.

\end{remark}
\begin{remark}
One can choose $\alpha_k = \alpha = \frac{1}{\beta(m+4)}$ to maximise the convergence rate, but this results in maximising the asymptotic optimality gap as well. 
\end{remark}

\subsubsection{Convergence of Proximal Gradient with forward gradients}
The proximal gradient algorithm, also known as forward-backward algorithm, applies to the optimisation problems where the cost function is the sum of two functions
\begin{align}\label{sum_problem}
    \min_{\boldsymbol{x}\in \mathbb{R}^m}\left\{\mathcal{L}(\boldsymbol{x}) := g(\boldsymbol{x}) + h(\boldsymbol{x})\right\}
\end{align}
which satisfies the following assumption.
\begin{assumption}\label{Proximal_assumptions}
Function $g$ is $\beta$-smooth and $h$ is a possibly nonsmooth convex function for all $k$.
\end{assumption}

The proximal PL-condition in Definition \ref{proximal PL} has been used to analyse the convergence of the generalised proximal gradient algorithm \citet{karimi2016linear} where authors have shown that the proximal PL condition is equivalent to the KL condition \citet{kurdyka1998gradients}. An important example of cost satisfying the proximal-PL inequality is the $\ell1$-regularized least squares problem.
\begin{lemma}\label{QG for proximal-PL}
Let $\mathcal{L}:\mathbb{R}^n \to \mathbb{R}$ be a function that satisfies Proximal-PL condition and Assumption \ref{Proximal_assumptions}. Then there exist a constant $\xi > 0$ such that for every $\boldsymbol{x}\in \mathbb{R}^n$ the following holds
\begin{align}
    \frac{\xi}{2}\|\boldsymbol{x} - \pi_{\mathcal{X}^*}(\boldsymbol{x}) \|^2 \leq \mathcal{L}(\boldsymbol{x}) - \mathcal{L}^*
\end{align}
\end{lemma}
\begin{lemma}
Consider the problem \eqref{sum_problem} where the Assumption \ref{Proximal_assumptions} holds and $\mathcal{L}$ satisfies the proximal-PL condition with parameter $\mu$. Then  for proximal gradient algorithm with the step size $\frac{1}{\beta}$ the following holds for some $C>0$
\begin{align*}
    \|\boldsymbol{x}_k - \pi_{\mathcal{X}^*}(\boldsymbol{x}_k) \| \leq C (1-\frac{\mu}{\beta})^{\frac{k}{2}}
\end{align*}
\end{lemma}
\begin{proof}
The proof is a result of Lemma \ref{QG for proximal-PL} and Theorem 5 in \citet{karimi2016linear}.
\end{proof}
\begin{remark} \label{Seq_length_Lemma}
The cost function does not need satisfy proximal-PL condition on the whole space. The condition should only be satisfied in a large enough ball around initial point, $B(\boldsymbol{x}_0, R)$ where $R \geq \frac{\sqrt{\frac{2}{\beta} (\mathcal{L}(\boldsymbol{x}_{0}) - \mathcal{L}^*)}}{1-\sqrt{1-\frac{\mu}{\beta}}}$.
\end{remark}

We are interested in analysing the performance of proximal gradient algorithm in an online setting using forward gradient 
\begin{align}\label{Prox_grad_iter}
    \boldsymbol{x}_{k+1} = \textit{prox}_{\gamma_k h_k} (\boldsymbol{x}_k - \gamma_k \boldsymbol{v}_k(\boldsymbol{x}_k, U_k)) := \argmin_{\boldsymbol{y}} h_k(\boldsymbol{y}) + \frac{1}{2\gamma_k}\|\boldsymbol{y} - (\boldsymbol{x}_k - \gamma_k \boldsymbol{v}_k(\boldsymbol{x}_k, U_k)) \|^2
\end{align}
where $\boldsymbol{v}_k(\boldsymbol{x}, U)$ is the forward gradient of function $g_k$ at point $\boldsymbol{x}$. We aim at analysing the algorithm above in terms of expected optimality gap.

\begin{assumption}\label{bounded subgrads}
The functions $g_k$ has bounded gradients and functions $h_k$ sub-gradients for each $k$, i.e. $\| \nabla g_k(\boldsymbol{x})\| \leq c_1$ and $\| \partial h_k(\boldsymbol{x}) \| \leq c_2$.
\end{assumption}

We analyse the performance of the algorithm above, \eqref{Prox_grad_iter}, under the proximal PL assumption for time-varying objective functions. We show that the tracking error of the algorithm linearly converges to a neighbourhood of zero, where the asymptotic error depend on the dimension.

\begin{theorem} \label{prox_fwd_grad_timevar}
    Assume that $\mathcal{L}_k$ is sum of a $\beta-$smooth function $g_k$ and a nonsmooth convex function $h_k$  which satisfies Proximal-PL condition with parameter $\mu$ and Assumptions \ref{drift} and \ref{bounded subgrads} hold. If we use proximal algorithm with forward gradient and constant step-size $\gamma_k = \frac{1}{\beta}$ and random direction $U_k$ is chosen from a standard normal distribution, then 
\begin{align*}
    \mathbb{E}\{\| \boldsymbol{x}_{k+1} - \pi_{\mathcal{X}_{k+1}^*}(\boldsymbol{x}_{k+1}) \|^2\} &\leq   \frac{2}{ \xi (1-(1-\frac{\mu}{\beta})^{\ell})} \bigg(\eta_0 + \eta^* + 2G_1\sqrt{(m+3)}  \bigg) + \frac{2}{\xi}(1-\frac{\mu}{\beta})^{k\ell} (\mathcal{L}_{0}(\boldsymbol{x}_{0})-\mathcal{L}_{0}^*)
\end{align*}
where $\xi$ is as in Lemma \ref{QG for proximal-PL} and $G_1 = \frac{2c_1 (c_1 + c_2)}{\beta}$ with $c_1$ and $c_2$ as in Assumption \ref{bounded subgrads}.
\end{theorem}

\begin{remark}
Under the hypotheses of Theorem \ref{prox_fwd_grad_timevar} the following holds:
\begin{align}
    \limsup_{k\to \infty} \mathbb{E} \bigg\{ \mathcal{L}_{k+1}(\boldsymbol{x}_{k+1})-\mathcal{L}_{k+1}^* \bigg\} \leq \frac{1}{1-(1-\frac{\mu}{\beta})^{\ell}} (\eta_0 + \eta^* + G_1\sqrt{(m+3)})
\end{align}
\end{remark}

\section{Illustrative Numerical Results}
The theoretical error bounds of previous sections are illustrated here employing a numerical example. We consider a sequence of time-varying linear least square problems of the form

\begin{align}
    \mathcal{L}_k(\boldsymbol{x}) = \frac{1}{2}\|A_k\boldsymbol{x} - \boldsymbol{b}_k \|^2
\end{align}
where $A_k\in \mathbb{R}^{n\times m}$ and $b_k \in \mathbb{R}^n$. The cost function $\mathcal{L}_k$ satisfies PL condition for every choice of $A_k$ and $\boldsymbol{b}_k$. We consider the over-parameterised case $m=60$ and $n = 10$. The vector $\boldsymbol{b}_k$ is generated as $\boldsymbol{b}_{k+1} = \boldsymbol{b}_k + \delta b_k$, where $\delta b_k$ follows a normal distribution $\mathcal{N}(\boldsymbol{0}, 10^{-2}I_n)$. The matrix $A_k$ is generated using its singular value decomposition, $A_k = U\Sigma_k V^T$ where $U \in \mathbb{R}^{n \times r}$ and $V \in \mathbb{R}^{m \times r}$ are orthogonal matrices and $\Sigma_{k+1} = \Sigma_k - 10^{-6}I_r$ that $\Sigma_0 = diag\ \{0.1, 0.2, \dots, 1\}$ and $r=10$. This setting assures that $\sup_{\boldsymbol{x}} \{\mathcal{L}_{k+1}(\boldsymbol{x}) - \mathcal{L}_k(\boldsymbol{x})\}$ is bounded. We have run 50 different experiments with the same starting point and have calculated the average loss. Results are plotted in Figure \ref{Tracking plot}, showing the performance of the actual algorithm and the derived theoretical bounds, validating the convergence results.

\begin{figure}[h]
    \centering
    \includegraphics[scale=0.7]{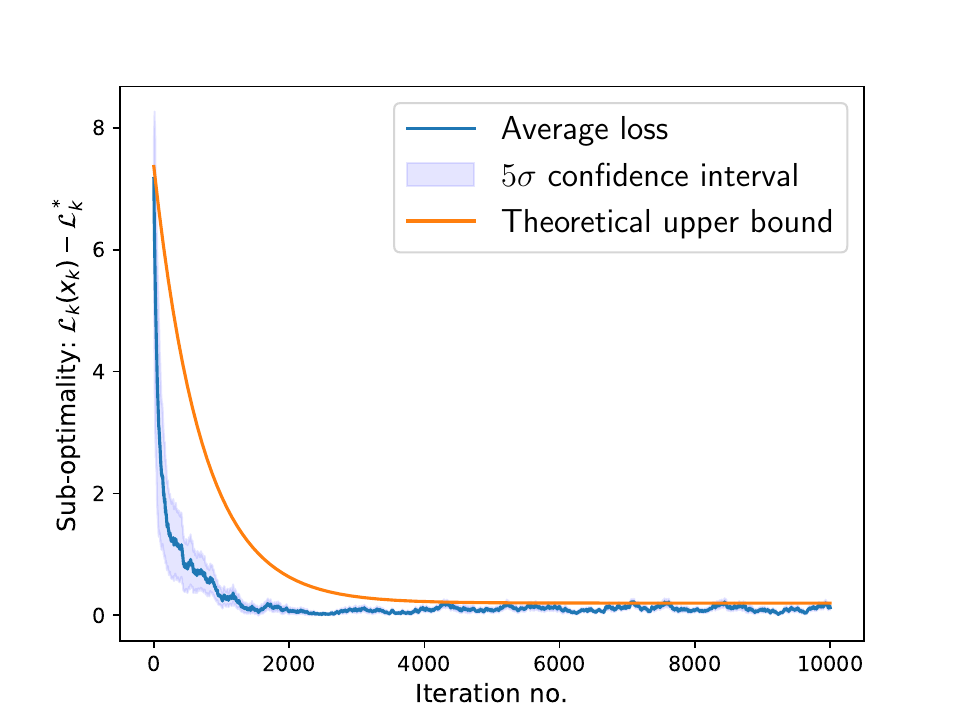}
    \caption{Performance of online gradient descent algorithm using directional derivatives, and the theoretical bound}
    \label{Tracking plot}
\end{figure}

\section{Conclusion}\label{conclusion}
In this paper, we analysed the performance of gradient-based first-order algorithms focusing on faster algorithms for calculating gradients, namely forward gradients. We have exploited the fact that non-linear least square problems in over-parameterised settings will satisfy the PL condition. Based on this observation, we proved that the (proximal-) gradient-based algorithm based on the forward mode of automatic differentiation (forward gradient) with nonconvex objective functions convergences to optimal value function at a linear rate. We have also analysed the convergence of these algorithms in a time-varying setting and proved the linear convergence to the neighbourhood of a global minimiser. This paper gives new insights on using the forward mode of automatic differentiation in problems with limited resources and fast-changing cost functions where calculating full gradients using the backpropagation algorithm is either impossible or inefficient.

\bibliography{references.bib}

\appendix

\section{Proof of Lemma \ref{lemma_fix_fwd_gd}} \label{proof_fix_fwd_gd}
\begin{proof}

Using descent lemma we have
\begin{align}
    \mathcal{L}(\boldsymbol{x}_{k+1}) &\leq \mathcal{L}(\boldsymbol{x}_k) + \langle \nabla \mathcal{L}(\boldsymbol{x}_k), \boldsymbol{x}_{k+1} - \boldsymbol{x}_k \rangle + \frac{\beta}{2}\|\boldsymbol{x}_{k+1} - \boldsymbol{x}_k \|^2 \\
    & =  \mathcal{L}(\boldsymbol{x}_k) -\alpha_k \langle \nabla \mathcal{L}(\boldsymbol{x}_k), \boldsymbol{v}_k(\boldsymbol{x}_k, U_k) \rangle + \frac{\beta}{2}\alpha_k^2\|\boldsymbol{v}_k(\boldsymbol{x}_k, U_k) \|^2
\end{align}
By taking conditional expectation given $\boldsymbol{x}_k$ with respect to $U_k$, and using \eqref{Exp_var_nest} we obtain
\begin{align*}
    \mathbb{E}\{\mathcal{L}(\boldsymbol{x}_{k+1}) \mid \boldsymbol{x}_k \} &\leq \mathcal{L}(\boldsymbol{x}_k) -\alpha_k  \|\nabla \mathcal{L}(\boldsymbol{x}_k) \|^2 +\frac{1}{2}\beta \alpha_k^2(m+4)\|\nabla \mathcal{L}(\boldsymbol{x}_k) \|^2\\
  & = \mathcal{L}(\boldsymbol{x}_k) -\underbrace{\alpha_k(1-\frac{\beta}{2}(m+4)\alpha_k)}_{:=\gamma_k}\|\nabla\mathcal{L}(\boldsymbol{x}_k) \|^2.
\end{align*}
We now fix $\alpha_k = \alpha = \frac{1}{\beta(m+4)}$ which implies $\gamma_k = \gamma = \frac{1}{2\beta(m+4)}$. Assuming that $\mathcal{L}$ satisfies $\mu-PL$ condition at $\boldsymbol{x}_k$, we have that
\begin{align}
        \mathbb{E}\{\mathcal{L}(\boldsymbol{x}_{k+1}) \mid \boldsymbol{x}_k \} &\leq \mathcal{L}(\boldsymbol{x}_k) -2\mu \gamma (\mathcal{L}(\boldsymbol{x}_k) - \mathcal{L}^*).
\end{align}

By using the tower rule of expectations and induction we have
\begin{align}
    \mathbb{E}\{\mathcal{L}(\boldsymbol{x}_{k+1}) - \mathcal{L}^*\} &\leq (\eta_0 + \eta^*)  (1-2\mu\gamma)^k (\mathcal{L}(\boldsymbol{x}_0) - \mathcal{L}_0^*) \\
    &\leq  (1-2\mu\gamma)^k (\mathcal{L}(\boldsymbol{x}_0) - \mathcal{L}_0^*).
\end{align}
Invoking  $\gamma = \frac{1}{2\beta(m+4)}$ completes the proof.

\end{proof}
\section{Proof of Lemma \ref{QG for proximal-PL}}\label{QG for proximal-PL proof}

Let $\boldsymbol{y} : = \argmin_{\boldsymbol{y}}\left\{ \langle \nabla g(\boldsymbol{x}), \boldsymbol{y}-\boldsymbol{x} \rangle + \frac{\beta}{2}\|\boldsymbol{y}-\boldsymbol{x} \|^2 + h(\boldsymbol{y}) - h(\boldsymbol{x})\right\} = \textit{prox}_{\frac{1}{\beta}h}(\boldsymbol{x} - \frac{1}{\beta} \nabla g(\boldsymbol{x}))$,
\begin{align*}
        \mathcal{L}(\boldsymbol{y}) &= g(\boldsymbol{y}) + h(\boldsymbol{y}) +h(\boldsymbol{x}) - h(\boldsymbol{x}) \\
    &\leq g(\boldsymbol{x}) + h(\boldsymbol{y}) + \langle \nabla g(\boldsymbol{x}), \boldsymbol{y} - \boldsymbol{x} \rangle + \frac{\beta}{2}\|\boldsymbol{y} - \boldsymbol{x}\|^2 +h(\boldsymbol{x}) - h(\boldsymbol{x}) \\
    & \stackrel{(a)}{=} \mathcal{L}(\boldsymbol{x}) -\frac{1}{2\beta} \mathcal{D}_h(\boldsymbol{x}, \beta)
\end{align*}
where in (a) we have used the optimality condition and definition of Proximal-PL condition. Therefore, 
\begin{align}
    \mathcal{D}_{h}(\boldsymbol{x} , \beta) \leq 2\beta ( \mathcal{L}(\boldsymbol{x}) - \mathcal{L}^* )
\end{align}
Using the optimiality condition we have $\boldsymbol{y} - \boldsymbol{x} = -\frac{1}{\beta} \big( \nabla g(\boldsymbol{x}) + \boldsymbol{s} \big)$ where $\boldsymbol{s}\in \partial h(\boldsymbol{y})$. Plugging this back into \eqref{prox PL ineq} gives

\begin{align*}
     \mathcal{D}_h(\boldsymbol{x}, \beta) &= -2\beta \bigg(\langle \nabla g(\boldsymbol{x}), \boldsymbol{y} - \boldsymbol{x} \rangle + \frac{\beta}{2}\|\boldsymbol{y} - \boldsymbol{x}\|^2 +h(\boldsymbol{y}) - h(\boldsymbol{x}) \bigg)\\
     &= -2\beta \bigg(\langle -\beta(\boldsymbol{y} - \boldsymbol{x}) - \boldsymbol{s}, \boldsymbol{y} - \boldsymbol{x} \rangle + \frac{\beta}{2}\|\boldsymbol{y} - \boldsymbol{x}\|^2 +h(\boldsymbol{y}) - h(\boldsymbol{x}) \bigg)\\
     &=-2\beta \bigg( -\frac{\beta}{2}\|\boldsymbol{y} - \boldsymbol{x} \|^2 - \langle \boldsymbol{s}, \boldsymbol{y} - \boldsymbol{x} \rangle  +h(\boldsymbol{y}) - h(\boldsymbol{x}) \bigg)\\
     &= \beta^2 \|\boldsymbol{y} - \boldsymbol{x} \|^2 + 2\beta \big(h(\boldsymbol{x}) - h(\boldsymbol{y}) -\langle \boldsymbol{s}, \boldsymbol{x} - \boldsymbol{y} \rangle \big)\\
     &\stackrel{(a)}{\geq} \beta^2 \|\boldsymbol{y} - \boldsymbol{x} \|^2
\end{align*}
where in (a) we have used the convexity of $h$. 

Using the fact that the conditions Proximal-EB (error bounds) and Proximal-PL are equivalent (see Appendix G in \citet{karimi2016linear}) there exist $c>0$ such that 
\begin{align}
    \|\boldsymbol{x} - \pi_{\mathcal{X}^*}(\boldsymbol{x}) \| \leq c \left\|\boldsymbol{x} - \textit{prox}_{\frac{1}{\beta}h}(\boldsymbol{x} - \frac{1}{\beta} \nabla g(\boldsymbol{x})) \right\|
\end{align}
combining the last three inequalities yields in
\begin{align}
    2\beta ( \mathcal{L}(\boldsymbol{x}) - \mathcal{L}^* ) \geq \mathcal{D}_{h}(\boldsymbol{x} , \beta) \geq \beta^2 \|\boldsymbol{y} - \boldsymbol{x} \|^2 \geq \frac{\beta^2}{c^2} \| \boldsymbol{x} - \pi_{\mathcal{X}^*}(\boldsymbol{x}) \|^2
\end{align}
setting $\xi = \frac{\beta}{c^2}$ completes the proof.

\section{Proof of Remark \ref{Seq_length_Lemma} } \label{remark seq length proof}
By using Lipschitz continuity of gradient of $g$ we can see
\begin{align*}
    \mathcal{L}(\boldsymbol{x}_{k+1}) &= g(\boldsymbol{x}_{k+1}) + h(\boldsymbol{x}_{k+1}) +h(\boldsymbol{x}_{k}) - h(\boldsymbol{x}_{k}) \\
    &\leq g(\boldsymbol{x}_{k}) + \langle \nabla g(\boldsymbol{x}_{k}), \boldsymbol{x}_{k+1} - \boldsymbol{x}_{k} \rangle + \frac{\beta}{2}\|\boldsymbol{x}_{k+1} - \boldsymbol{x}_{k}\|^2 +h(\boldsymbol{x}_{k}) - h(\boldsymbol{x}_{k}) \\
    & \stackrel{(a)}{=} \mathcal{L}(\boldsymbol{x}_{k}) -\frac{1}{2\beta} \mathcal{D}_h(\boldsymbol{x}_{k}, \beta)
\end{align*}
where in (a) we have used the Proximal PL definition \eqref{prox PL ineq}. Therefore
\begin{align} \label{upper_D}
    \mathcal{D}_h(\boldsymbol{x}_{k}, \beta) \leq 2 \beta  \big( \mathcal{L}(\boldsymbol{x}_{k}) -  \mathcal{L}(\boldsymbol{x}_{k+1}) \big) \leq 2 \beta  \big( \mathcal{L}(\boldsymbol{x}_{k}) -  \mathcal{L}^* \big).
\end{align}
Using the optimiality condition for \eqref{prox PL ineq} results in that $\boldsymbol{x}_{k+1} - \boldsymbol{x}_{k} = -\frac{1}{\beta} \big( \nabla g(\boldsymbol{x}_{k}) + \boldsymbol{s}^{k+1} \big)$ where $\boldsymbol{s}^{k+1}\in \partial h(\boldsymbol{x}_{k+1})$. Plugging this back into \eqref{prox PL ineq} we have

\begin{align*}
     \mathcal{D}_h(\boldsymbol{x}_{k}, \beta) &= -2\beta \bigg(\langle \nabla g(\boldsymbol{x}_{k}), \boldsymbol{x}_{k+1} - \boldsymbol{x}_{k} \rangle + \frac{\beta}{2}\|\boldsymbol{x}_{k+1} - \boldsymbol{x}_{k}\|^2 +h(\boldsymbol{x}_{k+1}) - h(\boldsymbol{x}_{k}) \bigg)\\
     &= -2\beta \bigg(\langle -\beta(\boldsymbol{x}_{k+1} - \boldsymbol{x}_{k}) - \boldsymbol{s}^{k+1}, \boldsymbol{x}_{k+1} - \boldsymbol{x}_{k} \rangle + \frac{\beta}{2}\|\boldsymbol{x}_{k+1} - \boldsymbol{x}_{k}\|^2 +h(\boldsymbol{x}_{k+1}) - h(\boldsymbol{x}_{k}) \bigg)\\
     &=-2\beta \bigg( -\frac{\beta}{2}\|\boldsymbol{x}_{k+1} - \boldsymbol{x}_{k} \|^2 - \langle \boldsymbol{s}^{k+1}, \boldsymbol{x}_{k+1} - \boldsymbol{x}_{k} \rangle  +h(\boldsymbol{x}_{k+1}) - h(\boldsymbol{x}_{k}) \bigg)\\
     &= \beta^2 \|\boldsymbol{x}_{k+1} - \boldsymbol{x}_{k} \|^2 + 2\beta \big(h(\boldsymbol{x}_{k}) - h(\boldsymbol{x}_{k+1}) -\langle \boldsymbol{s}^{k+1}, \boldsymbol{x}_{k} - \boldsymbol{x}_{k+1} \rangle \big)\\
     &\stackrel{(a)}{\geq} \beta^2 \|\boldsymbol{x}_{k+1} - \boldsymbol{x}_{k} \|^2
\end{align*}
where in (a) we have used the convexity of $h$. By combining the last inequality with \eqref{upper_D} we have
\begin{align}
    \|\boldsymbol{x}_{k+1} - \boldsymbol{x}_{k} \| &\leq \sqrt{\frac{2}{\beta}\big( \mathcal{L}(\boldsymbol{x}_{k}) -  \mathcal{L}^* \big)} \\
    & \leq \sqrt{\frac{2}{\beta} (\mathcal{L}(\boldsymbol{x}_{0}) - \mathcal{L}^*)}(1-\frac{\mu}{\beta})^{\frac{k}{2}} 
\end{align}
which in turn gives a bound on the whole length of the sequence
\begin{align}
    \|\boldsymbol{x}_{k+1}- \boldsymbol{x}_0 \| \leq \sum_{i=0}^{k}  \|\boldsymbol{x}_{i+1} - \boldsymbol{x}_{i} \| \leq \frac{\sqrt{\frac{2}{\beta} (\mathcal{L}(\boldsymbol{x}_{0}) - \mathcal{L}^*)}}{1-\sqrt{1-\frac{\mu}{\beta}}}:=R
\end{align}

\section{Proof of Theorem \ref{prox_fwd_grad_timevar}} \label{prox_fwd_grad_timevar proof}

Using Assumption \ref{drift} we have
\begin{align*}
    \mathcal{L}_{k+1}(\boldsymbol{x}_{k+1}) &\leq \eta_0 + g_k(\boldsymbol{x}_{k+1}) + h_k(\boldsymbol{x}_{k+1}) + h_k(\boldsymbol{x}_{k}) - h_k(\boldsymbol{x}_{k})\\
    & \stackrel{(a)}{\leq} \eta_0 + g_k(\boldsymbol{x}_{k}) + h_k(\boldsymbol{x}_{k})+ \langle \nabla g_k(\boldsymbol{x}_{k}), \boldsymbol{x}_{k+1} - \boldsymbol{x}_{k} \rangle + \frac{\beta}{2}\|\boldsymbol{x}_{k+1}- \boldsymbol{x}_{k} \|^2 + h_k(\boldsymbol{x}_{k+1})- h_k(\boldsymbol{x}_{k})\\
    & = \eta_0 + \mathcal{L}_k(\boldsymbol{x}_{k}) + h_k(\boldsymbol{x}_{k+1})+ \langle \nabla g_k(\boldsymbol{x}_{k}), \boldsymbol{x}_{k+1} - \boldsymbol{x}_{k} \rangle + \frac{\beta}{2}\|\boldsymbol{x}_{k+1}- \boldsymbol{x}_{k} \|^2 - h_k(\boldsymbol{x}_{k})\\
    &= \eta_0 + \mathcal{L}_k(\boldsymbol{x}_{k}) + h_k(\boldsymbol{x}_{k+1})+ \langle \boldsymbol{v}_k(\boldsymbol{x}_k, U_k) +\nabla g_k(\boldsymbol{x}_{k})  - \boldsymbol{v}_k(\boldsymbol{x}_k, U_k), \boldsymbol{x}_{k+1} - \boldsymbol{x}_{k} \rangle \\
    &+ \frac{\beta}{2}\|\boldsymbol{x}_{k+1}- \boldsymbol{x}_{k} \|^2 - h_k(\boldsymbol{x}_{k})\\
    &\stackrel{(b)}{=} \eta_0 + \mathcal{L}_k(\boldsymbol{x}_{k}) - \frac{1}{2\beta} \widetilde{\mathcal{D}}_{h_k}(\boldsymbol{x}_k, \beta, U_k) + 
    \langle \nabla g_k(\boldsymbol{x}_{k})  - \boldsymbol{v}_k(\boldsymbol{x}_k, U_k), \boldsymbol{x}_{k+1} - \boldsymbol{x}_{k} \rangle
\end{align*}
where in (a) we have used the smoothness of $g_k$ and in (b) we have used the following definition
\begin{align}
    \widetilde{\mathcal{D}}_h(\boldsymbol{x}, \alpha, U): = -2\alpha \min_{\boldsymbol{y}} \left\{ \langle \boldsymbol{v}({\boldsymbol{x}}, U), \boldsymbol{y}-\boldsymbol{x} \rangle + \frac{\alpha}{2}\|\boldsymbol{y}-\boldsymbol{x} \|^2 + h(\boldsymbol{y}) - h(\boldsymbol{x})\right\}.
\end{align}
By adding and subtracting term we have
\begin{align*}
    \mathcal{L}_{k+1}(\boldsymbol{x}_{k+1})-\mathcal{L}_{k+1}^* &\leq \eta_0 + \eta^* + \mathcal{L}_k(\boldsymbol{x}_{k}) - \mathcal{L}_k^* - \frac{1}{2\beta} {\mathcal{D}}_{h_k}(\boldsymbol{x}_k, \beta)  \\
    &+ \frac{1}{2\beta} {\mathcal{D}}_{h_k}(\boldsymbol{x}_k, \beta) - \frac{1}{2\beta} \widetilde{\mathcal{D}}_{h_k}(\boldsymbol{x}_k, \beta, U_k) + \langle \nabla g_k(\boldsymbol{x}_{k})  - \boldsymbol{v}_k(\boldsymbol{x}_k, U_k), \boldsymbol{x}_{k+1} - \boldsymbol{x}_{k} \rangle \\
    &\stackrel{(a)}{\leq} \eta_0 + \eta^* + (1-\frac{\mu}{\beta})(\mathcal{L}_k(\boldsymbol{x}_k) - \mathcal{L}_k^*) + \frac{1}{2\beta}\left(\widetilde{\mathcal{D}}_{h_k}(\boldsymbol{x}_k, \beta, U_k) - \mathcal{D}_{h_k}(\boldsymbol{x}_k, \beta) \right) \\
    & + \langle \nabla g_k(\boldsymbol{x}_{k})  - v(\boldsymbol{x}_k, U_k), \boldsymbol{x}_{k+1} - \boldsymbol{x}_{k} \rangle \\
    &\stackrel{(b)}{\leq} \eta_0 + \eta^* + (1-\frac{\mu}{\beta})(\mathcal{L}_k(\boldsymbol{x}_k) - \mathcal{L}_k^*)\\
    &+ \min_{\boldsymbol{y}} \left\{ \langle \nabla g_k(\boldsymbol{x}_k), \boldsymbol{y}-\boldsymbol{x}_k \rangle + \frac{\beta}{2}\|\boldsymbol{y}-\boldsymbol{x}_k \|^2 + h_k(\boldsymbol{y}) - h_k(\boldsymbol{x}_k)\right\}\\
    & - \min_{\boldsymbol{y}} \left\{ \langle \boldsymbol{v}_k({\boldsymbol{x}_k}, U_k), \boldsymbol{y}-\boldsymbol{x}_k \rangle + \frac{\beta}{2}\|\boldsymbol{y}-\boldsymbol{x}_k \|^2 + h(\boldsymbol{y}) - h(\boldsymbol{x}_k)\right\}\\
    & + \langle \nabla g_k(\boldsymbol{x}_{k})  - \boldsymbol{v}_k(\boldsymbol{x}_k, U_k), \boldsymbol{x}_{k+1} - \boldsymbol{x}_{k} \rangle\\
    &\stackrel{(c)}{\leq} \eta_0 + \eta^* + (1-\frac{\mu}{\beta})(\mathcal{L}_k(\boldsymbol{x}_k) - \mathcal{L}_k^*) +2 \langle \nabla g_k(\boldsymbol{x}_{k})  - \boldsymbol{v}_k(\boldsymbol{x}_k, U_k), \boldsymbol{x}_{k+1} - \boldsymbol{x}_{k} \rangle 
\end{align*}
where in (a) we have used \eqref{prox PL ineq}, in (b) we have used the definitions for $\widetilde{\mathcal{D}}_{h_k}$ and $\mathcal{D}_{h_k}$, and (c) is true because of the update step and sub-optimality condition.

Note that form optimality condition of \eqref{Prox_grad_iter} we have
\begin{align*}
    \boldsymbol{x}_{k+1}- \boldsymbol{x}_{k}\in -\frac{1}{\beta}\left({\boldsymbol{v}_k(\boldsymbol{x}_{k}, U_k) + \partial h_k(\boldsymbol{x}_{k+1})}\right).
\end{align*}
Therefore, $\|\boldsymbol{x}_{k+1} - \boldsymbol{x}_{k} \|$ is bounded because of boundedness of (sub)gradients, $\|\boldsymbol{x}_{k+1} - \boldsymbol{x}_{k} \| \leq  \frac{c_1 + c_2}{\beta}$. Using this in the last inequality above results in
\begin{align}
    \mathcal{L}_{k+1}(\boldsymbol{x}_{k+1})-\mathcal{L}_{k+1}^* \leq \eta_0 + \eta^* +  (1-\frac{\mu}{\beta})( \mathcal{L}_{k}(\boldsymbol{x}_{k})-\mathcal{L}_{k}^*) + 2\frac{c_1 + c_2}{\beta}\|\nabla g_k(\boldsymbol{x}_{k})  - \boldsymbol{v}_k(\boldsymbol{x}_k, U_k) \|.
\end{align}
By taking conditional expectation given $\boldsymbol{x}_k$ with respect to $U_k$, we have
\begin{align*}
    \mathbb{E} \bigg\{ \mathcal{L}_{k+1}(\boldsymbol{x}_{k+1})-\mathcal{L}_{k+1}^* \mid \boldsymbol{x}_k \bigg\} &\leq \eta_0 + \eta^* + (1-\frac{\mu}{\beta}) (\mathcal{L}_{k}(\boldsymbol{x}_{k})-\mathcal{L}_{k}^*) \\
    &+ 2\frac{c_1 + c_2}{\beta}\ \mathbb{E} \bigg\{ \|\nabla g_k(\boldsymbol{x}_{k})  - \boldsymbol{v}_k(\boldsymbol{x}_k, U_k) \| \mid \boldsymbol{x}_k \bigg\}.
\end{align*}
Invoking Lyapunov's inequality, $(\mathbb{E}\|Y\|^r)^{1/r}\leq (\mathbb{E}\|Y\|^s)^{1/s}$ for $r=1, s=2$, we have
\begin{align*}
    \mathbb{E} \bigg\{ \mathcal{L}_{k+1}(\boldsymbol{x}_{k+1})-\mathcal{L}_{k+1}^* \mid \boldsymbol{x}_k \bigg\} &\leq \eta_0 + \eta^* + (1-\frac{\mu}{\beta}) (\mathcal{L}_{k}(\boldsymbol{x}_{k})-\mathcal{L}_{k}^*) \\
    &+ 2\frac{c_1 + c_2}{\beta}\ \sqrt{ \mathbb{E} \bigg\{ \|\nabla g_k(\boldsymbol{x}_{k})  - \boldsymbol{v}_k(\boldsymbol{x}_k, U_k) \|^2 \mid \boldsymbol{x}_k \bigg\}}\\
    & \stackrel{(a)}{=} \eta_0 + \eta_* + (1-\frac{\mu}{\beta}) (\mathcal{L}_{k}(\boldsymbol{x}_{k})-\mathcal{L}_{k}^*) + 2\frac{c_1 + c_2}{\beta}\sqrt{(m+3)} \| \nabla g_k(\boldsymbol{x}_k) \|\\
    & \stackrel{(b)}{\leq} (1-\frac{\mu}{\beta}) (\mathcal{L}_{k}(\boldsymbol{x}_{k})-\mathcal{L}_{k}^*) + \eta_0 + \eta^* + 2c_1\frac{c_1 + c_2}{\beta}\sqrt{(m+3)},
\end{align*}
where in (a)  and (b) we have used \eqref{Var-Dir_Deriv} and the boundedness of gradients, respectively. Using tower rule of expectations
\begin{align*}
    \mathbb{E} \bigg\{ \mathcal{L}_{k+1}(\boldsymbol{x}_{k+1})-\mathcal{L}_{k+1}^*  \bigg\} &\leq (1-\frac{\mu}{\beta})^{k} (\mathcal{L}_{0}(\boldsymbol{x}_{0})-\mathcal{L}_{0}^*) \\
    &+ \frac{1}{1-(1-\frac{\mu}{\beta})} \bigg(\eta_0 + \eta^* + 2c_1 \frac{c_1 + c_2}{\beta}\sqrt{(m+3)}  \bigg),
\end{align*}
Similarly, with $\ell$ iterations at each step we have
\begin{align*}
    \mathbb{E} \bigg\{ \mathcal{L}_{k+1}(\boldsymbol{x}_{k+1})-\mathcal{L}_{k+1}^*  \bigg\} &\leq (1-\frac{\mu}{\beta})^{k\ell} (\mathcal{L}_{0}(\boldsymbol{x}_{0})-\mathcal{L}_{0}^*) \\
    &+ \frac{1}{1-(1-\frac{\mu}{\beta})^{\ell}} \bigg(\eta_0 + \eta^* + 2c_1 \frac{c_1 + c_2}{\beta}\sqrt{(m+3)}  \bigg),
\end{align*}
Invoking Lemma \ref{QG for proximal-PL} completes the proof.

\end{document}